\documentclass[10pt,letter,reqno]{article}

\usepackage{color}
\definecolor{red}{rgb}{1,0,0}
\definecolor{green}{rgb}{0,1,0}
\definecolor{blue}{rgb}{0,0,1}

\definecolor{refkey}{gray}{.625}
\definecolor{labelkey}{gray}{.625}

\let\oldmarginpar\marginpar
\renewcommand\marginpar[1]{\-\oldmarginpar[\raggedleft\footnotesize #1]%
{\raggedright\footnotesize #1}}

\usepackage[pdftitle={Equivariant Dixmier-Douady Classes},
pdfauthor={Mathieu Stienon},pdfkeywords={keywords}]{hyperref}
\hypersetup{pdfpagemode=UseNone,colorlinks=false,pdfpagelayout=SinglePage,pdfstartview=FitH}

\usepackage[textwidth=135mm,textheight=245mm,left=2.5cm,
verbose]{geometry}

\usepackage[T1]{fontenc}
\usepackage{ae,aecompl,aeguill}
\usepackage[latin1]{inputenc}

\usepackage{setspace}
\usepackage{indentfirst}

\usepackage[intlimits]{amsmath}
\usepackage{amssymb}
\usepackage{stmaryrd}

\usepackage{graphicx}
\usepackage[all]{xy}
\def\dar[#1]{\ar@<2pt>[#1]\ar@<-2pt>[#1]}

\usepackage{amsthm}

\theoremstyle{plain}
\newtheorem{prop}{Proposition}[section]
\newtheorem{lem}[prop]{Lemma}

\newtheorem{cor}[prop]{Corollary}
\newtheorem{thm}[prop]{Theorem}

\newtheorem*{prop*}{Proposition}
\newtheorem*{lem*}{Lemma}
\newtheorem*{sublem*}{Sublemma}
\newtheorem*{cor*}{Corollaire}
\newtheorem*{thm*}{Th\'eor\`eme}
\newtheorem*{hypo*}{Hypothesis}
\newtheorem*{question*}{Question}
\newtheorem*{conjecture*}{Conjecture}
\newtheorem*{scholum*}{Scholum}
\newtheorem{defn}[prop]{Definition}
\newtheorem*{defn*}{D\'efinition}

\newtheoremstyle{slanted}
  {3pt}
  {3pt}
  {\slshape}
  {}
  {\bfseries}
  {.}
  {.5em}
  {}

\theoremstyle{slanted}

\newtheorem*{example*}{Example}
\newtheorem*{examples*}{Examples}

\newtheorem*{ex*}{Example}
\newtheorem*{exs*}{Examples}

\newtheorem*{remark*}{Remark}
\newtheorem*{remarks*}{Remarks}
\newtheorem{rmk}[prop]{Remark}

\newtheorem*{rmk*}{Remark}
\newtheorem*{rmks*}{Remarks}

\theoremstyle{definition}

\newtheorem*{con*}{Construction}
\newtheorem*{note*}{Note}

\theoremstyle{remark}

\newtheorem*{warning*}{Warning}
\newtheorem*{shortnote*}{Note}
\newtheorem*{claim*}{Claim}
\newtheorem*{axiom*}{Axiom}

\DeclareMathOperator{\ad}{ad}
\DeclareMathOperator{\Ad}{Ad}

\DeclareMathOperator{\id}{id}

\DeclareMathOperator{\Lie}{Lie}

\DeclareMathOperator{\pr}{pr}

\newcommand{\beq}[1]{\begin{equation}\label{#1}}
\newcommand{\eeq}{\end{equation}}

\newcommand{\RR}{\mathbb{R}}

\newcommand{\ZZ}{\mathbb{Z}}

\newcommand{\rond}{\circ}

\newcommand{\genrel}[2]{\left\{ #1 | #2 \right\}}

\newcommand{\XX}{\mathfrak{X}} 
\newcommand{\OO}{\Omega} 

\newcommand{\toto}{\rightrightarrows}

\newcommand{\isomorphism}{\cong}

\newcommand{\inv}{^{-1}}
\newcommand{\simplicial}{_{\scriptscriptstyle\bullet}}
\newcommand{\graded}{^{\scriptscriptstyle\bullet}}

\newcommand{\mfg}{\mathfrak{g}} 
\newcommand{\mfgs}{\mathfrak{g}^*} 

\newcommand{\eqcls}[1]{\left[#1\right]}

\newcommand{\ddtz}[1]{\big.\tfrac{d}{dt}#1\big|_0}

\newcommand{\gm}{\Gamma}

\newcommand{\ii}{\mathbin{\vrule width1.5ex height.4pt\vrule height1.5ex}}
\newcommand{\ip}{\ii}
\newcommand{\lie}[2]{[#1,#2]} 

\newcommand{\xto}[1]{\xrightarrow{#1}}
\newcommand{\gmt}{\widetilde{\gm}}
\newcommand{\prt}{\widetilde{\pr}}
\newcommand{\delt}{\widetilde{\del}}
\newcommand{\source}{t}
\newcommand{\target}{s}
\newcommand{\sourcer}{{\source^\rtimes}}
\newcommand{\targetr}{{\target^\rtimes}}
\newcommand{\del}{\partial}
\newcommand{\dels}{\del^\times}
\newcommand{\delr}{\del^\rtimes}
\newcommand{\delrt}{\delt^\rtimes}
\newcommand{\Ds}{D^\times}

\newcommand{\Drt}{\widetilde{D}^\rtimes}
\newcommand{\dg}{d_G}

\newcommand{\act}{\star}
\newcommand{\pp}{p}
\DeclareMathOperator{\CURV}{curv}
\newcommand{\curv}{\CURV_G(\theta)}
\newcommand{\curvv}{\CURV_G(\theta')}
\newcommand{\OOG}{\OO_G}
\newcommand{\etag}{\eta_G}
\newcommand{\bg}{B_G}
\newcommand{\bb}{B}
\newcommand{\qq}{Q}
\newcommand{\ff}{f}
\newcommand{\lambdap}{\lambda'}
\newcommand{\pg}{p_G}
\DeclareMathOperator{\BS}{\psi}
\DeclareMathOperator{\DR}{DR}
\newcommand{\pih}{\widehat{\pi}}
\newcommand{\gam}{\gamma}

\newcommand{\xtilde}{\widetilde{x}}
\newcommand{\ytilde}{\widetilde{y}}
\newcommand{\ztilde}{{\widetilde{z}}}

\newcommand{\be }{\begin{eqnarray*}}
\newcommand{\ee }{\end{eqnarray*}}

\newcommand{\tx}{\tilde{x}}
\newcommand{\ty}{\tilde{y}}
\newcommand{\gmgerbe}{\gmt\xto{\pp}\gm\toto X}
\newcommand{\gmgerbee}{\gmt'\xto{\pp'}\gm'\toto X'}
\newcommand{\gmG}{\gmt\rtimes G\xto{\pg}\gm\rtimes G\toto X}
\newcommand{\gmGG}{\gmt'\rtimes G\xto{\pg'}\gm'\rtimes G\toto X'}

\let\ceV=\overleftarrow

\newcommand{\TXL}{MR2119241}
\newcommand{\TX}{TX}
\newcommand{\Jeffrey}{MR1321064}
\newcommand{\Crainic}{MR2016690}
\newcommand{\Hitchin}{MR1876068}
\newcommand{\BCWZ}{MR2068969}
\newcommand{\Giraud}{MR0344253}
\newcommand{\Segal}{MR0232393}
\newcommand{\BX}{BX}
\newcommand{\Chatterjee}{Chatterjee}
\newcommand{\Murray}{MR1405064}

\newcommand{\Meinrenken}{MR2026898}
\newcommand{\MS}{MR1794295}
\newcommand{\Brylinski}{MR1197353}
\newcommand{\Dupont}{MR0500997}
\def\polhk#1{\setbox0=\hbox{#1}{\ooalign{\hidewidth
    \lower1.5ex\hbox{`}\hidewidth\crcr\unhbox0}}} 
\newcommand{\Gawedzki}{MR1945806}
\newcommand{\Mathai}{MR1977885}

\renewcommand{\epsilon}{\varepsilon}

\begin{document}


\title{\textsc{Equivariant Dixmier-Douady Classes}}
\author{
\textsc{Mathieu Sti\'enon}
\thanks{E.T.H.~Z\"urich, Departement Mathematik, R\"amistrasse 101, 8092 Z\"urich, Switzerland, 
\href{mailto:stienon@math.ethz.ch}{\texttt{stienon@math.ethz.ch}}} 
\thanks{Research supported by the European Union through the FP6 Marie Curie R.T.N. ENIGMA
(Contract number MRTN-CT-2004-5652).} 
}
\date{}
\maketitle

\begin{abstract}
An equivariant bundle gerbe \`a la Meinrenken over a $G$-manifold $M$ is known to be 
a special type of $S^1$-gerbe over the differentiable stack $[M/G]$. 
We prove that the natural morphism relating the Cartan and simplicial models of equivariant cohomology in degree 3 maps the Dixmier-Douady class of an equivariant bundle gerbe \`a la Meinrenken to the Behrend-Xu-Dixmier-Douady class of the corresponding $S^1$-gerbe.
\end{abstract}

\tableofcontents

\section{Introduction}\label{sec:1}

Recently, following Brylinski's pioneering work \cite{\Brylinski}, there has been increasing interest in studying the differential geometry of gerbes. In particular, Murray defined and investigated bundle gerbes \cite{\Murray}, which were further studied by Chatterjee \cite{\Chatterjee} and Hitchin \cite{\Hitchin}. 

By definition, a bundle gerbe over a smooth manifold $M$ is a central $S^1$-extension 
of the groupoid $X\times_M X\toto X$ coming from a surjective submersion $X\xto{\pi}M$. 
A class in $H^3(M,\ZZ)$ is associated to any bundle gerbe over $M$. It is called the Dixmier-Douady (DD) class. The bundle gerbes over $M$ are classified, up to Morita equivalence (or stable equivalence in \cite{MR1794295}), by their DD classes.  

Moreover, like the Chern classes of $S^1$-bundles, the DD classes can be expressed, up to torsion elements, in terms of the $3$-curvature. 
The equivariant counterparts of bundle gerbes are called equivariant bundle gerbes \cite{\Mathai}. 
They are $G$-equivariant central $S^1$-extensions of a groupoid $X\times_M X\toto X$ associated to a $G$-equivariant surjective submersion $X\xto{\pi}M$. 
Meinrenken \cite{\Meinrenken} and Gaw\polhk edzki-Reis \cite{\Gawedzki} studied extensively the equivariant bundle gerbes over simple Lie groups. 
In \cite{\Meinrenken}, Meinrenken introduced the equivariant 3-curvature of an equivariant bundle gerbe. It is a closed equivariant 3-form in Cartan's model of equivariant cohomology. Its cohomology class corresponds to the equivariant DD class of the gerbe. 

Recently, Behrend-Xu studied $S^1$-gerbes over differentiable stacks. 
From their perspective, a $G$-equivariant bundle gerbe is an $S^1$-gerbe over the quotient stack $[M/G]$, which is a Morita equivalence class of central $S^1$-extensions of groupoids $\tilde{H}\to H\toto N$, where $H\toto N$ is Morita equivalent to the groupoid $M\rtimes G\toto M$. 
From connection type data on such a central $S^1$-extension, Behrend-Xu construct a characteristic class in the degree 3 cohomology de~Rham cohomology group $H_{\DR}^3(H\simplicial)$ of the corresponding simplicial manifold. 

The purpose of this paper is to establish an explicit connection between equivariant bundle gerbes \`a la Meinrenken and $S^1$-gerbes over $[M/G]$ \`a la Behrend-Xu. 
For this purpose, we use an explicit map, obtained by Bursztyn-Crainic-Weinstein-Zhu \cite{\BCWZ}, between the Cartan and simplicial models of equivariant cohomology in degree 3. As a byproduct, we establish some further properties of the BCWZ-map and show that it is indeed an isomorphism at the cohomology level. We hope that this result will be of independent interest. 




Our main theorem states that
the BCWZ isomorphism maps the Meinrenken equivariant DD class of a $G$-equivariant bundle gerbe over a $G$-manifold $M$ to the Behrend-Xu-Dixmier-Douady (BXDD) class of the corresponding $S^1$-gerbe over the quotient stack $[M/G]$.


The paper is organized as follows.

Section~\ref{sec:2.1} recalls the definition of equivariant bundle gerbes and equivariant central $S^1$-extensions, while Section~\ref{sec:2.2} recalls how the Dimier-Douady (DD) class of a $G$-equivariant bundle gerbe may be computed from connection type data. 

Section~\ref{sec:3.1} gives a brief account of $S^1$-gerbes $\tilde{\XX}$ over a differentiable stack $\XX$ and their DD classes. The DD class of an $S^1$-gerbe induces a degree 3 de~Rham cohomology class called BXDD class which can be computed from connection type data.

In Section~\ref{sec:3.2}, we explain how an equivariant bundle gerbe over a $G$-manifold $M$ (in the sense of Murray and Meinrenken) produces an $S^1$-gerbe over the stack $[M/G]$ (in the sense of Behrend-Xu). And we compute the BXDD class of the central $S^1$-extension of groupoids presenting the $S^1$-gerbe over $[M/G]$ associated to a $G$-equivariant bundle gerbe over the manifold $M$.

In Section~\ref{sec:4.1}, we discuss the explicit formula due to BCWZ relating the Cartan and simplicial models of equivariant cohomology in degree 3 and detail some additional properties of this map. 

Sections~\ref{sec:4.2}~and~\ref{sec:4.3} contain our main result.

\paragraph{Acknowledgements}

The author is indebted to Jean-Louis Tu and Ping Xu for providing access to their unpublished manuscript \cite{\TX}.

\paragraph{Preliminaries}

We start by recalling a few definitions and conventions used throughout this paper. 

Let $G$ be a compact connected Lie group. 
A $G$-manifold is a smooth manifold $M$ endowed with a right action of $G$, which is denoted $M\times G:(m,g)\mapsto m\act g$.
If $M$ is a $G$-manifold and $\xi\in\mfg=\Lie(G)$, then $\widehat{\xi}$ denotes the 
infinitesimal vector field on $M$ defined by the relation 
$\widehat{\xi}|_x=\ddtz{x\act e^{t\xi}}, \qquad \forall x\in M$. 

The Cartan model for equivariant cohomology is the differential complex 
$\big(\OOG\graded(M),\dg\big)$ defined by 
\[ \OOG^k(M)=\bigoplus_{2i+j=k}\big(S^i\mfgs\otimes\OO^j(M)\big)^G \] 
and \[ \big(\dg\alpha\big)(\xi)=d\big(\alpha(\xi)\big)-\widehat{\xi}\ip\alpha(\xi) ,\] 
where $\xi\in\mfg$ and the element $\alpha$ of $\OOG\graded(M)$ is seen as an 
$\OO\graded(M)$-valued polynomial on $\mfg$.

The multiplication of a Lie groupoid $\gm_1\toto\gm_0$ is denoted by $\gm_2\to\gm_1:(x,y)\mapsto x\cdot y$, where $\gm_2:=\genrel{(x,y)\in\gm_1\times\gm_1}{\source(x)=\target(y)}$. 

By a $G$-groupoid, we mean a Lie groupoid $\gm_1\toto\gm_0$ such that both $\gm_1$ and $\gm_0$ are $G$-manifolds 
and all the structure maps ($\target,\source,m,\iota,\epsilon$) are $G$-equivariant. 
Recall that any Lie groupoid $\gm_1\rightrightarrows \gm_0$ gives rise to a simplicial manifold 
\begin{equation*} 
\xymatrix{ \cdots \ar@<-1.5ex>[r]\ar@<-.5ex>[r]\ar@<.5ex>[r]\ar@<1.5ex>[r] & \gm_2
\ar@<1ex>[r]\ar[r]\ar@<-1ex>[r] & \gm_1\ar@<-.5ex>[r]\ar@<.5ex>[r] & \gm_0 }
\end{equation*}
where \[ \gm_n=\{(x_1,\dots,x_n)|\source(x_i)=\target(x_{i+1}),\; i=1,\dots,n-1\} \]
is the set of composable $n$-tuples of elements of $\gm_1$ and the face maps 
$\epsilon^n_i:\gm_n\to \gm_{n-1}$ are given, for $n>1$, by 
\begin{align*} & \epsilon_0^n(x_1,x_2,\dots,x_n)=(x_2,\dots,x_n) \\
& \epsilon_n^n(x_1,x_2,\dots,x_n)=(x_1,\dots,x_{n-1}) \\
& \epsilon_i^n(x_1,x_2,\dots,x_n)=(x_1,\dots,x_i x_{i+1},\dots,x_n), \; \; 1\le i \le n-1 ,
\end{align*} and, for $n=1$, by $\epsilon_0^1(x)= \source(x)$ and $\epsilon_1^1(x)=\target(x)$. 
They satisfy the simplicial relations
\[ \epsilon_i^{n-1}\rond\epsilon_j^n=\epsilon_{j-1}^{n-1}\rond\epsilon_i^n 
\qquad \forall i<j .\]
See \cite{\Segal,\Dupont} for more details. 

Given a Lie groupoid $\gm_1\toto \gm_0$, 
consider the double complex $\OO\graded(\gm\simplicial)$:
\begin{equation*}
\xymatrix{ \cdots&\cdots&\cdots&\\
\OO^2(\gm_0)\ar[u]^d\ar[r]^\partial &\OO^2(\gm_1)\ar[u]^d\ar[r]^\partial 
&\OO^2(\gm_2)\ar[u]^d\ar[r]^\partial &\cdots \\
\OO^1(\gm_0)\ar[u]^d\ar[r]^\partial &\OO^1(\gm_1)\ar[u]^d\ar[r]^\partial 
&\OO^1(\gm_2)\ar[u]^d\ar[r]^\partial &\cdots \\
\OO^0(\gm_0)\ar[u]^d\ar[r]^\partial &\OO^0(\gm_1)\ar[u]^d\ar[r]^\partial 
&\OO^0(\gm_2)\ar[u]^d\ar[r]^\partial &\cdots }
\end{equation*}
Its coboundary maps are $d: \OO^k(\gm_p) \to \OO^{k+1}(\gm_p)$, the usual exterior 
differential of smooth forms 
and $\del: \OO^k(\gm_p) \to \OO^k(\gm_{p+1})$, the alternating sum 
of the pullbacks by the face maps: 
\beq{partial} \del=\sum_{i=0}^{n} (-1)^i (\epsilon_i^n)^* .\eeq
We denote the total differential by $D= (-1)^p d+\del$.
The cohomology groups 
\[ H_{\DR}^k(\gm\simplicial)
:=H^k\big(\OO\graded(\gm\simplicial),D\big) \]
of the total complex $\big(\OO\graded_{\DR}(\gm\simplicial),D\big)$
(where $\OO_{\DR}^k(\gm\simplicial)
=\bigoplus_{i=0}^k\OO^{k-i}(\gm_i)$)
are called the \emph{de~Rham cohomology groups of the groupoid $\gm_1\toto\gm_0$}.

Now let $\gm_1\toto\gm_0$ be a $G$-groupoid. We can consider the double complex $\OOG\graded(\gm\simplicial)$: 
\begin{equation*}
\xymatrix{ \cdots &\cdots &\cdots & \\
\OOG^2(\gm_0)\ar[u]^\dg\ar[r]^\del &\OOG^2(\gm_1)\ar[u]^\dg\ar[r]^\del &\OOG^2(\gm_2)\ar[u]^\dg\ar[r]^\del &\cdots \\
\OOG^1(\gm_0)\ar[u]^\dg\ar[r]^\del &\OOG^1(\gm_1)\ar[u]^\dg\ar[r]^\del &\OOG^1(\gm_2)\ar[u]^\dg\ar[r]^\del &\cdots \\
\OOG^0(\gm_0)\ar[u]^\dg\ar[r]^\del &\OOG^0(\gm_1)\ar[u]^\dg\ar[r]^\del &\OOG^0(\gm_2)\ar[u]^\dg\ar[r]^\del &\cdots }
\end{equation*}
Its coboundary operators are $\dg: \OOG^k(\gm_p) \to \OOG^{k+1}(\gm_p)$, the differential operator of the Cartan model 
and $\del: \OOG^k(\gm_p) \to \OOG^k(\gm_{p+1})$, 
the natural extension of \eqref{partial}.
We denote the total differential by $D_G=(-1)^p \dg+\del$.
The cohomology groups \[ H_G^k(\gm\simplicial):=
H^k\big(\OOG\graded(\gm\simplicial),D_G\big) \] of the total complex 
are called the \emph{equivariant cohomology groups of the $G$-groupoid $\gm_1\toto\gm_0$}. 
See \cite{\Jeffrey}.

\section{Equivariant bundle gerbes \`a la Meinrenken}
\label{sec:2}

In this section, we recall the notion of equivariant bundle gerbes and their equivariant Dixmier-Douady classes in terms of the Cartan model. We closely follow Meinrenken's approach \cite{\Meinrenken}.

\subsection{Equivariant central $S^1$-extensions}
\label{sec:2.1}

Assume that $X\xto{\pi}M$ is a surjective submersion. Consider the Lie groupoid 
\beq{gm} \gm\toto X , \qquad \text{with } \gm=X\times_M X ,\eeq 
the source and target maps $\target(x,y)=x$ and $\source(x,y)=y$, and the multiplication \[ (x,y)\cdot(y,z)=(x,z) .\] Then we have the Morita morphism \cite{\BX} 
\beq{gmM} \xymatrix{ \gm \dar[d] \ar[r]^{\pi'} & M \dar[d] \\ X \ar[r]^\pi& M } \eeq
where $\pi':\gm\to M$ is the map $(x,y)\mapsto\pi(x)=\pi(y)$. Indeed $\gm\toto X$ is the pullback of the trivial groupoid  $M\toto M$ to $X$ through $\pi$. 

Furthermore, if $G$ is a Lie group, $X$ and $M$ are $G$-manifolds and $X\xto{\pi}M$ is a $G$-equivariant surjective submersion, 
it is clear that the Lie group $G$ acts on $\gm\toto X$ by groupoid automorphisms, i.e. $\gm\toto X$ is a $G$-groupoid, and that $\pi'$ in \eqref{gmM} is a $G$-equivariant Morita morphism.

Recall that a \textbf{central $S^1$-extension} of a Lie groupoid $H\toto N$ consists of a morphism of Lie groupoids 
\beq{eq:p} \xymatrix{ \tilde{H} \dar[d] \ar[r]^{p} & H \dar[d] \\ N \ar[r]_{\id} & N } \eeq
and a left $S^1$-action on $\tilde{H}$, making $p:\tilde{H}\to H$ a (left) principal $S^1$-bundle \cite{\TXL,\BX}.
These two structures are compatible in the following sense: 
\[ (\lambda_1\tx)\cdot(\lambda_2\ty)=\lambda_1\lambda_2(\tx\cdot\ty) ,\]
for all $\lambda_1,\lambda_2\in S^1$ and $(\tx,\ty)\in\tilde{H}_2:=
\tilde{H}\times_{\source,N,\target}\tilde{H}$.

We will use the shorthand notation $\tilde{H}\xto{p}H\toto N$ to denote the above central $S^1$-extension.

A central $S^1$-extension $\tilde{H}\xto{p}H\toto N$ is said to be \textbf{$G$-equivariant} if both $\tilde{H}\toto N$ and $H\toto N$ are $G$-groupoids, the groupoid morphism $p:\tilde{H}\to H$ in \eqref{eq:p} is $G$-equivariant and the $G$-action preserves the principal $S^1$-bundle $\tilde{H}\xto{p} H$. That is, if the following relations: 
\begin{gather*}
(\tx\cdot\ty)\act g=(\tx\act g)\cdot(\ty\act g) \\ 
\pp(\tx\act g)=\pp(\tx)\act g  \\ 
(\lambda\tx)\act g=\lambda(\tx\act g) 
\end{gather*}
are satisfied for all $g\in G$, all composable pairs $(\tx,\ty)$ in $\tilde{H}_2$ and all $\lambda\in S^1$.



Bundle gerbes were invented by Murray \cite{\Murray} (see also \cite{\Hitchin,\Chatterjee}). By definition, a \textbf{bundle gerbe} over a manifold $M$ is a central $S^1$-extension of the Lie groupoid $\gm\toto X$ (as in \eqref{gm}) obtained from a surjective submersion $X\to M$. There is a natural equivalence relation on central $S^1$-extensions \cite{\BX,\TXL}, the so-called Morita equivalence (or stable equivalence in \cite{\MS}), whose equivalence classes are classified by the cohomology group $H^3(M,\ZZ)$. The class in $H^3(M,\ZZ)$ attached to a central $S^1$-extension is called its \textbf{Dixmier-Douady class}.
Equivariant bundle gerbes are equivariant counterparts of bundle gerbes. According to Meinrenken \cite{\Meinrenken}, a \textbf{$G$-equivariant bundle gerbe} over a $G$-manifold $M$ is a $G$-equivariant central $S^1$-extension of the groupoid $\gm\toto X$ associated to a $G$-equivariant surjective submersion $X\to M$ as in \eqref{gm}.

\subsection{Equivariant Dixmier-Douady classes}
\label{sec:2.2}

Below we recall Meinrenken's definition of the equivariant 3-curvature and equivariant Dixmier-Douady class of a $G$-equivariant bundle gerbe.\footnote{What Meinrenken called an ``equivariant connection" \cite{\Meinrenken} consists of both an equivariant connection and an equivariant curving in our
terminology.}

\begin{defn}\label{curving}
Let $\gmt\xto{\pp}\gm\toto X$ be a $G$-equivariant bundle gerbe,
where $X\xto{\pi}M$ is a $G$-equivariant surjective submersion
and $\gm=X\times_M X$ is the resulting groupoid as in \eqref{gm}.
\begin{enumerate}
\item An \textbf{equivariant connection} is a $G$-invariant 1-form $\theta\in\OO^1(\gmt)^G$ such that $\theta$ is a connection 1-form for the principal $S^1$-bundle $\gmt\xto{\pp}\gm$ and satisfies
\[ \delt\theta=0 .\]
\item Given an equivariant connection $\theta$, an \textbf{equivariant curving} is a degree 2 element $\bg\in\OOG^2(X)$ such that 
\beq{2} \curv=\del\bg ,\eeq
where $\curv$ denotes the equivariant curvature of the $S^1$-principal bundle $\gmt\xto{\pp}\gm$, i.e. the element $\curv\in\OOG^2(\gm)$ characterized by the relation 
\beq{1} \dg\theta=\pp^*\curv .\eeq
\item Given an equivariant connection and an equivariant
curving $(\theta,\bg)$, the corresponding \textbf{equivariant 3-curvature} is the equivariant 3-form $\etag\in\OOG^3(M)$ such that 
\beq{3} \pi^*\etag=\dg\bg .\eeq
\end{enumerate}
Here the coboundary operators associated to the groupoids $\gm\toto X$ and $\gmt\toto X$ as in \eqref{partial} are denoted $\del$ and $\delt$ respectively. 
\end{defn}

The following result seems to be standard (see \cite{\Meinrenken,\TX}). 
However, we could not find a complete proof in the literature. 
For the sake of completeness, we will sketch a proof below.

\begin{prop}\label{independence} 
Let $\gmt\xto{\pp}\gm\toto X$ be a $G$-equivariant bundle gerbe over a $G$-manifold $M$.
\begin{enumerate}
\item Equivariant connections and curvings $(\theta,\bg)$ always exist.
\item The class $\eqcls{\etag}\in H_G^3(M)$ defined by
the equivariant 3-curvature is independent of the choice of $\theta$ and $\bg$.
\end{enumerate}
\end{prop}

We need the following lemma.

\begin{lem}\label{Murray}
\begin{enumerate}
\item Given a surjective submersion $\pi:X\to M$, the
sequence
\beq{fes} 0\to\Omega^k(M)\xto{\pi^*}
\OO^k(X)\xto{\del}
\OO^k(\gm)\xto{\del}
\OO^k(\gm_2)\xto{\del}\cdots \eeq
is exact.
\item Given two $G$-manifolds $X$ and $M$ and a $G$-equivariant surjective submersion $\pi:X\to M$, the sequence
\beq{ses} 0\to\Omega^k(M)^G\xto{\pi^*}
\OO^k(X)^G\xto{\del}
\OO^k(\gm)^G\xto{\del}
\OO^k(\gm_2)^G\xto{\del}\cdots \eeq
is exact.
\item 
Given a $G$-equivariant surjective submersion $\pi:X\to M$ between two $G$-manifolds $X$ and $M$, the sequence 
\beq{tes} 0\to\Omega^k_G(M)\xto{\pi^*}
\OOG^k(X)\xto{\del}
\OOG^k(\gm)\xto{\del}
\OOG^k(\gm_2)\xto{\del}\cdots \eeq
is exact.
\end{enumerate}
\end{lem}

\begin{proof}
(a) This was proved in \cite{\Murray}. 

(b) Since the face maps of the simplicial manifold 
\[ \xymatrix{
\cdots \ar@<1.5ex>[r]\ar@<0.5ex>[r]\ar@<-0.5ex>[r]\ar@<-1.5ex>[r] & 
\Gamma_2 \ar@<1ex>[r]\ar[r]\ar@<-1ex>[r] & 
\Gamma \ar@<-.5ex>[r]\ar@<.5ex>[r] &
X } \]
(and $X\xto{\pi}M$) are all $G$-equivariant, $R_g^*$ commutes with $\partial$ (and $\pi^*$). 
Hence \eqref{ses} is a subcomplex of \eqref{fes}. 
Now $\pi^*$ in \eqref{ses} is a restriction of $\pi^*$ in \eqref{fes}, which is injective. 
Therefore, $\pi^*$ in \eqref{ses} is injective.
Finally, take $\omega\in\OO^k(\Gamma_p)^G\subset\OO^k(\Gamma_p)$ 
such that $\del\omega=0$. By (a), there exists $\nu\in\OO^k(\Gamma_{p-1})$ such that $\del\nu=\omega$. Since the group $G$ is compact, 
we can choose a left-invariant Haar measure $dg$ on $G$ 
and define a $G$-invariant k-form $\nu'$ which satisfies $\del\nu'=\omega$ by 
\[ \nu'=\frac{1}{V}\int_G R_g^*\nu\; dg\in\OO^k(\Gamma_{p-1})^G ,\] 
where $V=\int_G 1\;dg$ is the volume of $G$. 

(c) For simplicity, we only consider the case $k=2$. 
Since $\pi$ is $G$-equivariant, it induces a pair of maps 
$\OO^2(M)^G\xto{\pi^*}\OO^2(X)^G$ and 
$\big(\mfgs\otimes\OO^2(M)\big)^G\xto{\pi^*}\big(\mfgs\otimes\OO^2(X)\big)^G$ and thus also 
$\OOG^2(M)\xto{\pi^*}\OOG^2(X)$. 
Because the face maps are $G$-equivariant, the alternate sum of their pullbacks induces the pair of maps 
\begin{gather*} 
\OO^2(\Gamma_{p-1})^G\xto{\partial}\OO^2(\Gamma_p)^G \\ 
\big(\mfgs\otimes\OO^0(\Gamma_{p-1})\big)^G\xto{\partial}
\big(\mfgs\otimes\OO^0(\Gamma_p)\big)^G 
,\end{gather*}
whose direct sum is the desired map 
\[ \OOG^2(\Gamma_{p-1})\xto{\partial}\OOG^2(\Gamma_p) .\]
Since \eqref{tes} is the direct sum of \eqref{ses} with $k=2$ (which is exact) and 
\[ \big(\mfgs\otimes\OO^0(M)\big)^G \xto{\pi^*} 
\big(\mfgs\otimes\OO^0(X)\big)^G \xto{\del} 
\big(\mfgs\otimes\OO^0(\Gamma)\big)^G \xto{\del} 
\big(\mfgs\otimes\OO^0(\Gamma_2)\big)^G \xto{\del} 
\cdots ,\]
it suffices to prove that the latter sequence is exact.
Let $f$ be an arbitrary element of $\big(\mfgs\otimes\OO^0(\Gamma_p)\big)^G$, i.e. 
\[ f(\Ad_g\xi)\;(x_1,x_2,\dots,x_p)=f(\xi)\;(x_1\act g,x_2\act g,\dots,x_p\act g) ,\] 
for all $\xi\in\mfg$, $g\in G$ and $(x_1,\dots,x_p)\in\gm_p$. 
And assume that $\del f=0$. 
Choose a basis $(e_1,\dots,e_n)$ of $\mfg$. Then $f(e_i)\in\OO^0(\Gamma_p)$ and 
$\del\big(f(e_i)\big)=0$. By (a), there exists $h(e_i)\in\OO^0(\Gamma_{p-1})$ such that 
$\del\big(h(e_i)\big)=f(e_i)$. 
We can define $h'(e_i)\in\big(\mfgs\otimes\OO^0(\Gamma_{p-1})\big)^G$ by 
\[ h'(e_i)=\frac{1}{V}\int_G R_g^*h(\Ad_{g\inv}e_i)\; dg .\] 
Clearly, $f(e_i)=\del\big(h'(e_i)\big)$ 
and thus \[ f(\sum_i e_i\xi^i)=\del\big(\sum_i h'(e_i)\xi^i\big) ,\] 
where $\sum_i h'(e_i)\xi^i\in\big(\mfgs\otimes\OO^0(\Gamma_{p-1})\big)^G$.
\end{proof}

\begin{proof}[Proof of Proposition~\ref{independence}]
(a) Take any connection 1-form $\theta'\in\Omega^1(\gmt)$ for the $S^1$-principal bundle $\pp:\gmt\to\gm$. Since $G$ is compact, one can always take $\theta'$ to be $G$-invariant, i.e. $\theta'\in\Omega^1(\gmt)^G$. It is simple to see that $\delt\theta'$ must be the pull back of a $G$-invariant 1-form on $\gm_2$ under $\pp:\gmt_2\to\gm_2$. That is $\delt\theta'=\pp^*\alpha$, where $\alpha\in\Omega^1(\gm_2)^G$. It follows from $\delt^2=0$ that $\del\alpha=0$. By Lemma~\ref{Murray}(b), we have $\alpha=\del A$ for some $A\in\Omega^1(\gm)^G$. Therefore $\theta=\theta'-\pp^*A$ is an equivariant connection.

Given an equivariant connection $\theta$, \eqref{1} implies that $\del\curv=0$ since $\delt\theta=0$. By Lemma~\ref{Murray}(c), there exists $\bg\in \Omega^2_G(X)$ such that $\curv=\del\bg$. That is, $\bg$ is an equivariant curving.
 
Assume that $(\theta',\bg',\etag)$ is another such triple. We have $\theta-\theta'=\pp^*\beta$ for some $\beta\in\Omega^1(\gm)^G$.
And $\del\beta=0$. By Lemma~\ref{Murray}(b), we have $\beta=\del\gamma$ for some $\gamma\in\Omega^1(X)^G$. Now
\begin{align*}
0=& \dg\theta-\dg\theta'-\pp^*\dg\beta \\
=& \pp^*\big(\curv-\curvv-\dg\del\gamma\big) \\ 
=& \pp^*\big(\del(\bg-\bg'-\dg\gamma)\big)
.\end{align*}
Therefore $\del(\bg-\bg'-\dg\gamma)=0$. Note that $\bg-\bg'-\dg\gamma\in\Omega_G^2(X)$. Hence, by Lemma~\ref{Murray}(c), there exists $\lambda\in\Omega_G^2(M)$ such that $\bg-\bg'-\dg\gamma=\pi^*\lambda$. Applying $\dg$ to both sides, we get $\dg(\bg-\bg'-\dg\gamma)=\pi^*\dg\lambda$, which implies that $\etag-\etag'=\dg\lambda$. The conclusion follows.
\end{proof}

The class $\eqcls{\etag}$ is called \textbf{equivariant Dixmier-Douady class} by Meinrenken \cite{\Meinrenken}.

\subsection{Morita equivalences}
\label{sec:2.3}

Recall that two central $S^1$-extensions $\tilde{H}\to H\toto N$ and $\tilde{H}'\to H'\toto N'$ are said to be Morita equivalent \cite{\BX,\TXL} if there exists a $\tilde{H}$-$\tilde{H}'$-bitorsor $Z$ endowed with a (left) $S^1$-action such that \[ (\lambda r)\cdot z\cdot r'=r\cdot(\lambda z)\cdot r'=r\cdot z\cdot(\lambda r') \] whenever $(\lambda,r,z,r')\in S^1\times\tilde{H}\times Z\times\tilde{H}'$ and the products make sense.

\begin{defn}
Two $G$-equivariant bundle gerbes $\gmgerbe$ and $\gmgerbee$ are Morita equivalent if they are Morita equivalent as central $S^1$-extensions, the equivalence bitorsor $Z$ is a $G$-space and 
\beq{bi} (r\cdot z\cdot r')\act g=(r \act g) \cdot (z \act g) \cdot (r'\act  g), \qquad\forall g\in G \eeq
whenever $(r,z,r')\in\gmt\times Z\times\gmt'$ and the products make sense. 
\end{defn}

A bitorsor satisfying \eqref{bi} is called a \textbf{$G$-equivariant bitorsor}.

\begin{prop}
If $\gmgerbe$ and $\gmgerbee$ are Morita equivalent $G$-equivariant bundle gerbes with equivalence bitorsor $Z$, the $S^1$-action on $Z$ is free and $Z/S^1$ is a $G$-equivariant $H$-$H'$-bitorsor. Hence $\gm\toto X$ and $\gm'\toto X'$ are Morita equivalent $G$-groupoids. In other words, the $G$-manifolds $M$ and $M'$ underlying the bundle gerbes $p$ and $p'$ are one and the same manifold.
\end{prop}

\section{Behrend-Xu-Dixmier-Douady classes}
\label{sec:3}

\subsection{General theory}
\label{sec:3.1}

In \cite{\BX} (see also \cite{\TXL}), Behrend-Xu developed a general theory of $S^1$-gerbes over differentiable stacks in terms of central $S^1$-extensions of Lie groupoids. Roughly speaking, an \textbf{$S^1$-gerbe $\widetilde{\XX}$ over a differentiable stack $\XX$} can be thought of as a Morita equivalence class of central $S^1$-extensions of Lie groupoids $H\toto N$, where $H\toto N$ is a presentation of the differentiable stack $\XX$. (One needs to choose a suitable representative amongst all presentations of the differentiable stack $\XX$, for not every presentation of the stack $\XX$ can be extended to a presentation of the stack $\widetilde{\XX}$. See \cite{\BX}.) According to Giraud \cite{\Giraud}, the $S^1$-gerbes over a differentiable stack $\XX$ are classified by the cohomology group $H^2(\XX,S^1)$. Hence, there exists a natural map
\[ \{\text{central $S^1$-extensions of $H\toto N$}\} \xto{\tau} H^2(\XX,S^1) .\]
Composing $\tau$ with the boundary map $H^2(\XX,S^1)\to H^3(\XX,\ZZ)$ associated to the short exact sequence 
\[ 0\to\ZZ\to \RR\xto{\exp}S^1\to 0 ,\] we get a map 
\[ \{\text{central $S^1$-extensions of $H\toto N$}\} \longrightarrow H^3(\XX,\ZZ)\isomorphism H^3(H\simplicial,\ZZ) .\] 
The image of a central $S^1$-extension under the above map is called its \textbf{Dixmier-Douady class} in \cite{\BX}.

Behrend-Xu also proved that, similarly to the Chern classes of bundles, the Dixmier-Douady classes of central $S^1$-extensions can be computed, up to torsion elements, from connection type data. 
Recall that a pseudo-connection on a central $S^1$-extension $\tilde{H}\xto{p}H\toto N$ is a sum 
\[ \theta+\lambda\in\Omega^1(\tilde{H})\oplus\Omega^2(N)\subset
\Omega_{\DR}^2(\tilde{H}\simplicial) \]
such that $\theta$ is a connection 1-form on the principal $S^1$-bundle $\tilde{H}\xto{p}H$ \cite{\BX}.
Its pseudo-curvature
\[ \eta+\omega+\Omega\in\Omega^1(H_2)\oplus\Omega^2(H)\oplus\Omega^3(N)\subset\Omega_{\DR}^3(H\simplicial) \] 
is defined by the relation 
\[ \tilde{D}(\theta+\lambda)=p^*(\eta+\omega+\Omega) .\]

\begin{thm}[\cite{\BX}]
The pseudo-curvature $\eta+\omega+\Omega$ is a 3-cocycle in $\Omega_{\DR}^3(H\simplicial)$. Its cohomology class $[\eta+\omega+\Omega]$ is an integer class in $H^3_{\DR}(H\simplicial)$, which is independent of the choice of pseudo-connection. Under the canonical homomorphism $H^3(H\simplicial,\ZZ)\to H^3_{DR}(H\simplicial)$, the Dixmier-Douady class of $\tilde{H}\xto{p}H\toto N$ maps to $[\eta+\omega+\Omega]$.
\end{thm}

The de~Rham class $[\eta+\omega+\Omega]\in H^3_{\DR}(H\simplicial)$ will be called \textbf{Behrend-Xu-Dixmier-Douady class}.

\subsection{An $S^1$-gerbe over $[M/G]$}
\label{sec:3.2}

There is a natural correspondence between equivariant bundle gerbes over a $G$-manifold $M$ in the sense of Murray and Meinrenken and $S^1$-gerbes over the stack $[M/G]$ in the sense of Behrend-Xu \cite{\BX,\TX,\TXL}. 

The quotient stack $[M/G]$ can be presented by the transformation groupoid $M\rtimes G\toto M$, where $\source(x,g)=xg$, $\target(x,g)=x$ and 
\[ (x,g)\cdot(y,h)=(x,gh), \qquad\text{when } y=x\act g .\]
Adopting the Behrend-Xu perspective, we note that an $S^1$-gerbe over the stack $[M/G]$ can always be presented by a central $S^1$-extension $\tilde{H}\to H\toto N$ of a Lie groupoid $H\toto N$ Morita equivalent to $M\rtimes G\toto M$.

Now consider, as in Section~\ref{sec:2}, a $G$-equivariant bundle gerbe $\gmt\xto{\pp}\gm\toto X$, where $X\xto{\pi}M$ is a $G$-equivariant surjective submersion and $\gm=X\times_M X$. Since $\gm\toto X$ (resp. $\gmt\toto X$) is a $G$-groupoid, we can form the transformation groupoid $\gm\rtimes G\toto X$ (resp. $\gmt\rtimes G\toto X$). 

If $\gm\simplicial=(\gm_1\toto\gm_0)$ is a $G$-groupoid, its transformation groupoid $\gm^{\rtimes}\simplicial=(\gm_1\rtimes G\toto\gm_0)$ is the groupoid whose source map is $\targetr:(\gam,g)\mapsto\target(\gam)$, whose target map is $\sourcer:(\gam,g)\mapsto\source(\gam)\act g$ and whose multiplication is given by
\[ (\gam_1,g_1)\cdot(\gam_2,g_2)=(\gam_1\cdot(\gam_2\act g_1\inv),g_1 g_2) ,\]
for any $\gam_1,\gam_2\in\gm_1$ and $g_1,g_2\in G$ such that $\source(\gam_1)\act g_1=\target(\gam_2)$.

\begin{lem}\label{lem:3.2}
\begin{enumerate}
\item The groupoids $\gm\rtimes G \toto X$ and $M\rtimes G\toto M$ are Morita equivalent.
\item Set $\pg(\widetilde{\gamma},g)=(\pp(\widetilde{\gamma}),g)$. Then $\gmt\rtimes G\xto{\pg}\gm\rtimes G\toto X$ is a central $S^1$-extension of Lie groupoids.
\end{enumerate}
\end{lem}

\begin{proof}
(a) Since $\gm\toto X$ is Morita equivalent to the trivial groupoid $M\toto M$ and the Morita morphism mapping $\gm\toto X$ to $M\toto M$ is $G$-equivariant, it follows that the morphism 
\beq{7} \xymatrix{ \gm\rtimes G \dar[d] \ar[r]^{\pih} & M\rtimes G \dar[d] \\
X \ar[r]_{\pi} & M } \eeq
defined by 
\[ \pih(x,y,g)=(\pi(y),g), \qquad \forall (x,y)\in\gm=X\times_M X \] 
is a Morita morphism of Lie groupoids.

(b) This follows immediately from the definition of $G$-equivariant central $S^1$-extension.
\end{proof}

Hence, the extension $\gmt\rtimes G\xto{\pg}\gm\rtimes G\toto X$ induces an $S^1$-gerbe over the quotient stack $[M/G]$ in the sense of Behrend-Xu.


We now compute the Behrend-Xu-Dixmier-Douady class of the central $S^1$-extension $\gmG$. 
Note that we have the following commutative diagram 
\[ \xymatrix{ \gmt\rtimes G \ar[r]^{\pg} \ar[d]_{\prt} & \gm\rtimes G \ar[d]^{\pr} \\ \gmt \ar[r]_{\pp} & \gm } \]
where $\prt$ and $\pr$ are mere projection maps rather than groupoid morphisms.

By $\delr$ and $\delrt$, we denote the coboundary operators associated to the Lie groupoids $\gm\rtimes G\toto X$ and $\gmt\rtimes G\toto X$, respectively, as in \eqref{partial}.

\begin{lem}\label{lem:3.3}
Assume that $\theta\in\Omega^{1}(\gmt)$ is a $G$-equivariant connection for the $G$-equivariant bundle gerbe $\gmgerbe$.
Then $\Theta:=\prt^*\theta\in\OO^1(\gmt\rtimes G)$ is a connection 1-form for the principal $S^1$-bundle $p_G:\gmt\rtimes G\to\gm\rtimes G$, hence a pseudo-connection for the central $S^1$-extension $\gmG$. One has 
\[ \delrt\Theta=\pg^*\zeta ,\] 
where $\zeta\in\OO^1\big((\gm\rtimes G)_2\big)$ is defined by 
\beq{15} (\pp_G^*\zeta)\big((v_{\xtilde},L_{g*}\xi),(w_{\ytilde},L_{h*}\eta)\big)=\widehat{\xi}_{\ytilde}\ip\theta ,\eeq
for any $\xi,\eta\in\mfg$, $g,h\in G$, $v_{\xtilde}\in T_{\xtilde}\gmt$ and $w_{\ytilde}\in T_{\ytilde}\gmt$ 
such that $\sourcer_*(v_{\xtilde},L_{g*}\xi)=\targetr_*(w_{\ytilde},L_{h*}\eta)$.
\end{lem}

\begin{proof}
Let $t\mapsto g_t$ and $t\mapsto h_t$ be paths in $G$ originating from $g$ and $h$ respectively and determining two vectors $\xi$ and $\eta$ of $\mfg$ by the relations $L_{g*}\xi=\ddtz{g_t}$ and $L_{h*}\eta=\ddtz{h_t}$.
Similarly, let $t\mapsto\xtilde_t$ and $t\mapsto\ytilde_t$ be smooth paths in $\gmt$ originating from $\xtilde$ and $\ytilde$, respectively, with $\ddtz{\xtilde_t}=v_{\xtilde}$ and $\ddtz{\ytilde_t}=w_{\ytilde}$ and such that, at any time $t$, the target of $\xtilde_t\act g_t$ coincides with the source of $\ytilde_t$. Then 
\begin{align*} 
& (\delrt\Theta)\big(\ddtz{(\xtilde_t,g_t)},\ddtz{(\ytilde_t,h_t)}\big) && \\ 
=\;& \Theta\big(\ddtz{(\ytilde_t,h_t)}\big)-\Theta\big(\ddtz{(\xtilde_t\cdot(\ytilde_t\act g_t\inv),g_t h_t)}\big) && \\ 
& +\Theta\big(\ddtz{(\xtilde_t,g_t)}\big) && \\ 
=\;& \theta\big(\ddtz{\ytilde_t}\big)-\theta\big(\ddtz{\xtilde_t\cdot(\ytilde_t\act g_t\inv)}\big)+\theta\big(\ddtz{\xtilde_t}\big) 
&& \text{since } \Theta=\pr^*\theta \\
=\;& \theta\big(\ddtz{\ytilde_t}\big)-\theta\big(\ddtz{\xtilde_t}\big)-\theta\big(\ddtz{\ytilde_t\act g_t\inv}\big)+\theta\big(\ddtz{\xtilde_t}\big) 
&& \text{since } \delt\theta=0 \\
=\;& \theta\big(\ddtz{\ytilde_t}\big)-\theta\big(\ddtz{(\ytilde_t\act g\inv)}\big)-\theta\big(\ddtz{(\ytilde\act g_t\inv )}\big) && \\ 
=\;& -\theta\big(\ddtz{y\act (g e^{t\xi})\inv}\big) && \text{since $\theta$ is $G$-invariant} \\ 
=\;& \theta\big(\widehat{\xi}_{\ytilde}\act g\inv\big) && \\ 
=\;& \theta\big(\widehat{\xi}_{\ytilde}\big) && \text{since $\theta$ is $G$-invariant} .
\end{align*}
The result follows.
\end{proof}

\begin{prop}
Let $\theta\in\OO^1(\gmt)$ be a $G$-equivariant connection on a $G$-equivariant bundle gerbe 
$\gmt\xto{\pp}\gm\toto X$ over a $G$-manifold $M$. Then
$\Theta:=\prt^*\theta\in\OO^1(\gmt\rtimes G)$ is a pseudo-connection
for the central $S^1$-extension $\gmG$. Its pseudo-curvature
is $\zeta-\pr^*\omega\in Z^3((\gm\times G)\simplicial)$, where $\zeta$
is given by \eqref{15} and $\omega$ is characterized by $d\theta=p^*\omega$.
Hence the Behrend-Xu-Dixmier-Douady class is 
$\eqcls{\zeta-\pr^*\omega}\in H_{\DR}^3\big((\gm\rtimes G)\simplicial\big)$.
\end{prop}

\begin{proof}
Since 
\[ d\Theta=d\;\prt^*\theta=\prt^*d\theta=\prt^*\pp^*\omega=\pg^*\pr^*\omega ,\] 
the associated pseudo-curvature is 
\[ \Drt\Theta=\delrt\Theta-d\Theta=\pg^*(\zeta-\pr^*\omega) .\]
\end{proof}

\begin{rmk}
From Lemma~\ref{lem:3.3}, we see that $\delrt\Theta$ vanishes if, and only if, $\delt\theta=0$ and $\theta$ is basic with respect to the $G$-action. In this case, $\Theta$ is a connection for the central $S^1$-extension $\gmG$ in the sense of Behrend-Xu \cite{\BX}. See \cite{\TX} for details.
\end{rmk}

\section{Linking Murray-Meinrenken to Behrend-Xu}
\label{sec:4}

\subsection{The BCWZ morphism}
\label{sec:4.1}

In order to compare Meinrenken's equivariant Dixmier-Douady class with the Behrend-Xu-Dixmier-Douady class, we need an explicit formula relating the Cartan and simplicial models of equivariant cohomology in degree 3. The following result can be found in \cite{\BCWZ} (though the group acts from the left in \cite{\BCWZ}).

\begin{prop}[Proposition~6.10 in \cite{\BCWZ}]\label{psi}
Let $N$ be a manifold on which a Lie group $G$ acts from the right. 
Consider the map 
\[ \BS:\OO_G^3(N)\to\OO_{\DR}^3\big((N\rtimes G)\simplicial\big) \] 
mapping $\alpha\in\OO^3(N)$ to itself ($\BS(\alpha)=\alpha$) and 
$\eta\in\big(\mfgs\otimes\OO^1(N)\big)^G$ to the 2-form 
$\BS(\eta)\in\OO^2(N\rtimes G)$ defined by 
\begin{equation}\label{17}
\BS(\eta)\;\big((v_1,L_{g*}\xi_1),(v_2,L_{g*}\xi_2)\big) =\eta(\xi_2)\;(v_1\act g)-\eta(\xi_1)\;(v_2\act g)+\eta(\xi_2)\;(\widehat{\xi_1}|_{x\act g})
,\end{equation}
for all $v_1,v_2\in T_x N$, $\xi_1,\xi_2\in\mfg$ and $g\in G$.
The map $\BS$ injects $Z^3(N)^G$ into $Z_{\DR}^3\big((N\rtimes G)\simplicial\big)$.
Moreover, \[ \BS\big(Z^3(N)^G\big)=Z_{\DR}^3\big((N\rtimes G)\simplicial\big)\cap 
\big(\OO^3(M)\oplus\OO^2(N\rtimes G)\big) .\]
\end{prop}

\begin{rmk}
One can check that the R.H.S. of \eqref{17} does indeed change sign when the indices 1 and 2 are permuted. 
\end{rmk}

From \eqref{17}, one easily deduce that, if $\sigma:N_1\to N_2$ is a $G$-equivariant map between two $G$-manifolds $N_1$ and $N_2$, then
\beq{eq:BSG}
\BS(\sigma^*\eta)=(\sigma\times 1)^*\BS(\eta), \qquad\forall\eta\in\big(\mfgs\otimes\OO^1(N_2)\big)^G.
\eeq

The following lemma will be needed later on.

\begin{lem}\label{lambda}
Given $\ff\in\big(\mfgs\otimes\OO^0(N)\big)^G$, then 
\beq{19} \BS(d\ff)=d\lambda ,\eeq 
where $\lambda\in\OO^1(N\times G)$ is defined by the relation
\[ \lambda(v_x,L_{g*}\xi)=\ff(\Ad_g\xi)(x), 
\qquad\forall v_x\in T_x N,\;\xi\in\mfg,\;g\in G .\]
\end{lem}

\begin{proof}
Since $\ff$ is $G$-equivariant, we have 
\beq{21} \ff(\Ad_g\xi)\;(x)=\ff(\xi)\;(x\act g) .\eeq
Take $g\in G$, $\xi_1,\xi_2\in\mfg$ and $v_1,v_2\in T_x N$.
Let $t\mapsto x_1(t)$ and $t\mapsto x_2(t)$ be two paths in $N$ originating from the same point $x$ such that $v_1=\ddtz{x_1(t)}$ and $v_2=\ddtz{x_2(t)}$.
From \eqref{17}, we get 
\begin{align*}
& \BS(d\ff)\;\big((v_1,L_{g*}\xi_1),(v_2,L_{g*}\xi_2)\big) && \\ 
=\;& \BS(d\ff)\;\big(\ddtz{(x_1(t),ge^{t\xi_1})},\ddtz{(x_2(t),ge^{t\xi_2})}\big) && \\
=\;& \ddtz{\ff(\xi_2)\;(x_1(t)\act g)}-\ddtz{\ff(\xi_1)\;(x_2(t)\act g)}
+\ddtz{\ff(\xi_2)\;(x\act (ge^{t\xi_1}))} && \\ 
=\;& \ddtz{\ff(\Ad_g\xi_2)\;(x_1(t))}-\ddtz{\ff(\Ad_g\xi_1)\;(x_2(t))}
+\ddtz{\ff(\Ad_g e^{t\ad_{\xi_1}}\xi_2)\;(x)} && \text{by \eqref{21}} 
.\end{align*}
But $\ff$ is linear in $\mfg$. Hence the last term is equal to $\ff(\Ad_g\lie{\xi_1}{\xi_2})\;(x)$. 

On the other hand, letting $\ceV{\xi_1}$ and $\ceV{\xi_2}$ be the 
left invariant vector fields on $G$ corresponding to $\xi_1$ and $\xi_2$, 
respectively, and choosing two vector fields $X_1$ and $X_2$ on $N$ such that 
$X_1|_x=v_1$ and $X_2|_x=v_2$, we obtain 
\begin{align*}
& (d\lambda )\big((v_1,L_{g*}\xi_1),(v_2,L_{g*}\xi_2)\big) \\ 
=\;& (v_1,L_{g*}\xi_1)\;\lambda(X_2,\ceV{\xi_2})-(v_2,L_{g*}\xi_2)\;\lambda(X_1,\ceV{\xi_1})
-\lambda(\lie{X_1}{X_2}_x,\lie{\ceV{\xi_1}}{\ceV{\xi_2}}_g) \\ 
=\;& \ddtz{\ff(\Ad_{ge^{t\xi_1}}\xi_2)\;(x_1(t))} - \ddtz{\ff(\Ad_{ge^{t\xi_2}}\xi_1)\;(x_2(t))} 
- \ff(\Ad_g\lie{\xi_1}{\xi_2})\;(x) \\ 
=\;& \ddtz{\ff(\Ad_{ge^{t\xi_1}}\xi_2)\;(x)}+\ddtz{\ff(\Ad_g\xi_2)\;(x_1(t))} 
- \ddtz{\ff(\Ad_{ge^{t\xi_2}}\xi_1)\;(x)} \\ 
& -\ddtz{\ff(\Ad_g\xi_1)\;(x_2(t))} - \ff(\Ad_g\lie{\xi_1}{\xi_2})\;(x) \\ 
=\;& \ddtz{\ff(\Ad_g\xi_2)\;(x_1(t))}-\ddtz{\ff(\Ad_g\xi_1)\;(x_2(t))}
+ \ff(\Ad_g\lie{\xi_1}{\xi_2})\;(x) .
\end{align*}
The result follows.
\end{proof}

\begin{lem}\label{lem:4.3}
Given $B\in\OO^2(N)^G$, define $\qq\in\big(\mfgs\otimes\OO^1(N)\big)^G$ by 
$\qq(\xi)=\widehat{\xi}\ip\bb$, $\forall\xi\in\mfg$.
Then $\BS(\qq)=-\dels B$, where $\dels:\OO^2(N)\to\OO^2(N\rtimes G)$
is the coboundary operator $\dels=\source^*-\target^*$
associated to the transformation groupoid $N\rtimes G\toto N$. 
\end{lem}

\begin{proof}
Note that for any $v_x\in T_xN$ and $\xi\in\mfg$,
\[ \target_*(v_x,L_{g*}\xi)=v_x \qquad \text{and} \qquad 
\source_*(v_x,L_{g*}\xi)=v_x\act g+\widehat{\xi}|_{x\act g} 
.\]
Thus, using the $G$-invariance of $\bb$, we obtain,
for any $v_1,v_2\in T_xN$, $\xi_1,\xi_2\in\mfg$ and $g\in G$,  
\begin{align*}
& (\dels B)\;\big((v_1,L_{g*}\xi_1),(v_2,L_{g*}\xi_2)\big) \\ 
=\;& \bb(\source_*(v_1,L_{g*}\xi_1),\source_*(v_2,L_{g*}\xi_2)) 
-\bb(\target_*(v_1,L_{g*}\xi_1),\target_*(v_2,L_{g*}\xi_2)) \\ 
=\;& \bb(\widehat{\xi_1}|_{x\act g}, v_2\act g) 
+ \bb(v_1\act g,\widehat{\xi_2}|_{x \act g}) 
+ \bb(\widehat{\xi_1}|_{x\act g},\widehat{\xi_2}|_{x\act g}) 
.\end{align*} 
On the other hand, we have 
\begin{align*}
& \BS(\qq)\;\big((v_1,L_{g*}\xi_1),(v_2,L_{g*}\xi_2)\big) \\ 
=\;& \qq(\xi_2)(v_1\act g)
- \qq(\xi_1)(v_2\act g)
+ \qq(\xi_2)(\widehat{\xi_1}|_{x\act g} ) \\ 
=\;& \bb(\widehat{\xi_2}|_{x\act g}, v_1\act g) 
- \bb(\widehat{\xi_1}|_{x\act g},v_2\act g) 
+ \bb(\widehat{\xi_2}|_{x\act g},\widehat{\xi_1}|_{x\act g}) 
.\end{align*}
The conclusion thus follows.
\end{proof}

As an immediate consequence, we have

\begin{cor}\label{cor:4.4}
The BCWZ map $\BS$ of Proposition~\ref{psi} maps the exact equivariant 3-forms $B_G^3(N)$ 
to the coboundaries $B_{\DR}^3\big((N\rtimes G)\simplicial\big)$. 
Therefore $\BS$ induces an isomorphism in cohomology: 
$H^3_G(N)\xto{\isomorphism}H^3_{\DR}\big((N\rtimes G)\simplicial\big)$.
\end{cor}

\begin{proof}
The 1-form $\lambda\in\OO^1(N\times G)$ defined in Lemma~\ref{lambda} satisfies 
\beq{dellam} \dels\lambda=0 .\eeq 
Indeed, since $\source_*(v_x,L_{g*}\xi)=\target_*(w_y,L_{h*}\eta)$ if, and only if, 
$y=x\act g$ and $w_y=v_x\act g+\widehat{\xi}_{x\act g}$, making use of the linearity in $\mfg$ and the $G$-equivariance of $f$ 
we obtain 
\begin{align*}
& \dels\lambda\big((v_x,L_{g*}\xi),(w_y,L_{h*}\eta)\big) \\ 
=\;& \lambda(w_y,L_{h*}\eta)-\lambda\big(v_x,L_{gh*}(\Ad_{h\inv}\xi+\eta)\big)+\lambda(v_x,L_{g*}\xi) \\ 
=\;& f(\Ad_h\eta)\;(y)-f\big(\Ad_{gh}(\Ad_{h\inv}\xi+\eta)\big)\;(x)+f(\Ad_g\xi)\;(x) \\ 
=\;& f(\Ad_h\eta)\;(x\act g)-f\big(\Ad_g(\Ad_h\eta)\big)\;(x) \\ 
=\;& 0.
\end{align*}
Since 
\[ \OO_G^2(N)=\OO^2(N)^G\oplus\big(\mfgs\otimes\OO^0(N)\big)^G, \] 
any element of $B_G^3(N)=\dg\big(\OO_G^2(N)\big)$ can be written as $\dg(\bb+\ff)$ for some 
$\bb\in\OO^2(N)^G$ and $\ff\in\big(\mfgs\otimes\OO^0(N)\big)^G$. 
By definition, $\dg\bb=d\bb-\qq$, where $\qq$ is defined as in Lemma~\ref{lem:4.3}, and $\dg\ff=d\ff$. 
Therefore, 
\begin{align*} 
\BS\big(\dg(\bb+\ff)\big) =& \BS(d\bb-\qq+d\ff) && \\ 
=& d\bb+\dels\bb+d\lambda && \text{by Lemma~\ref{lambda} and \ref{lem:4.3}} \\ 
=& \Ds(\bb+\lambda) && \text{by \eqref{dellam}} . 
\end{align*}
Thus $\BS$ not only maps closed equivariant 3-forms to cocycles of 
$\OO_{\DR}^3\big((N\rtimes G)\simplicial\big)$ (see Proposition~\ref{psi}) but also 
exact equivariant 3-forms to coboundaries of $\OO_{\DR}^3\big((N\rtimes G)\simplicial\big)$. 
Hence $\BS$ induces a homomorphism $H_G^3(N)\to H_{\DR}^3\big((N\rtimes G)\simplicial\big)$ in cohomology. 
Actually, the latter is an isomorphism, for $\BS$ is injective on the level of cocycles (according to Proposition~\ref{psi}) and any 3-cocycle in $\OO_{\DR}^3\big((N\rtimes G)\simplicial\big)$ is cohomologous to a 3-cocycle of the form 
$\alpha+\beta$, where $\alpha\in\OO^3(N)$ and $\beta\in\OO^2(N\times G)$. 
Indeed, since $N\rtimes G\toto N$ is a proper groupoid, the sequence 
\[ \OO^0\big((N\rtimes G)_2\big) \xto{\dels} 
\OO^0\big((N\rtimes G)_3\big) \xto{\dels} 
\OO^0\big((N\rtimes G)_4\big) \] 
is exact \cite[Proposition~1]{\Crainic} 
and, moreover, 
\[ \OO^1\big((N\rtimes G)_1\big) \xto{\dels} 
\OO^1\big((N\rtimes G)_2\big) \xto{\dels} 
\OO^1\big((N\rtimes G)_3\big) \] 
is also exact \cite[Lemma~1.5]{\TX}.
\end{proof}

\subsection{Main theorem}
\label{sec:4.2}

The Morita morphism 
\[ \xymatrix{ \gm\rtimes G \dar[d] \ar[r]^{\pih} & M\rtimes G \dar[d] \\ 
X \ar[r]_{\pi} & M } \]
as in \eqref{7} induces  a map of cochain complexes 
\[ \OO^k_{\DR}\big((M\rtimes G)\simplicial\big)\xto{\pih^*}\OO^k_{\DR}\big((\gm\rtimes G)\simplicial\big) \] 
which gives an isomorphism in cohomology: 
\[ H^k_{\DR}\big((M\rtimes G)\simplicial\big)\xto{\isomorphism}H^k_{\DR}\big((\gm\rtimes G)\simplicial\big) .\] 
By the symbol $\mu$, we will denote both the composition
\[ Z^3_G (M)\xto{\BS}Z_{\DR}^3\big((M\rtimes G)\simplicial\big)
\xto{\pih^*}Z^3_{\DR}\big((\gm\rtimes G)\simplicial\big) \]
and the induced isomorphism 
\[ H^3_G (M)\xto{\isomorphism}H^3_{\DR}\big((\gm\rtimes G)\simplicial\big) \] 
in cohomology. 

The main theorem can be stated as follows.
 
\begin{thm}\label{26}
Let $\gmt\xto{\pp}\gm\toto X$ be a $G$-equivariant bundle gerbe
over a $G$-manifold $M$ with equivariant connection $\theta$, equivariant
curving $\bg$ and equivariant 3-curvature $\etag$. 
Then \[ \mu\eqcls{\etag}=\eqcls{\zeta-\pr^*\omega} \] 
in $H_{\DR}^3\big((\gm\rtimes G)\simplicial\big)$.
In other words, the isomorphism $\mu$ maps Meinrenken's equivariant Dixmier-Douady class 
to the Behrend-Xu-Dixmier-Douady class. 
\end{thm}

Given two Morita equivalent $G$-equivariant bundle gerbes 
$\gmt\xto{\pp}\gm\toto X$ and $\gmt'\xto{\pp'}\gm'\toto X'$
over a $G$-manifold $M$, it is simple to see that
$\gmG$ and $\gmGG$ are Morita equivalent central $S^1$-extensions.
Hence they have isomorphic Behrend-Xu-Dixmier-Douady classes according
to \cite{\BX}. As a consequence of Theorem~\ref{26}, we have

\begin{cor}
Morita equivalent $G$-equivariant bundle gerbes have isomorphic equivariant Dixmier-Douady classes in Meinrenken's sense.
\end{cor}

\begin{rmk}
The above corollary asserts the existence of a map
\[ \left\{\text{\parbox{4.7cm}{Morita equivalence classes of
\mbox{$G$-equivariant} bundle gerbes over $M$}}\right\} \to H^3_G(M,\ZZ) ,\]
which is easily seen to be injective.
It is not clear though if this map is surjective
since requiring that a gerbe be $G$-equivariant may seem too strong 
(see Remark~2.8 in \cite{\Meinrenken}).
Recently, however, Tu-Xu proved that the above map is indeed
also surjective \cite{\TX}.
\end{rmk}

\subsection{Proof of the main theorem}
\label{sec:4.3}

First of all, let us wrap off the conditions
defining equivariant connections, equivariant curvings and equivariant
3-curvatures.
Since $\OOG^2(\gm)=\OO^2(\gm)^G\oplus\big(\mfgs\otimes\OO^0(\gm)\big)^G$, 
the equivariant curvature decomposes as 
\[ \curv=\omega+\phi ,\]
where $\omega\in\OO^2(\gm)^G$ and $\phi\in\big(\mfgs\otimes\OO^0(\gm)\big)^G$.
From \eqref{1}, we obtain 
\[ d\theta-\widehat{\xi}\ip\theta=\pp^*\omega+(\pp^*\phi)(\xi), \qquad \forall\xi\in\mfg ,\] 
which is equivalent to the pair of equations 
\begin{gather*}
d\theta=\pp^*\omega \\
-\widehat{\xi}\ip\theta=(\pp^*\phi)(\xi),\qquad\forall\xi\in\mfg 
\end{gather*}
Since $\OOG^2(X)=\OO^2(X)^G\oplus\big(\mfgs\otimes\OO^0(X)\big)^G$, 
the curving decomposes as 
\[ \bg=\bb+\ff ,\]
where $\bb\in\OO^2(X)^G$ and $\ff\in\big(\mfgs\otimes\OO^0(X)\big)^G$.
From \eqref{2}, we obtain 
\[ \omega+\phi=\del(\bb+\ff) \] 
or, equivalently,
\begin{gather}
\omega=\del\bb \label{888} \\
\phi=\del\ff \nonumber 
.\end{gather}
Hence $\delt\ff=\pp^*\del\ff=\pp^*\phi$ and 
\beq{9} (\delt\ff)(\xi)=(\pp^*\phi)(\xi)=-\widehat{\xi}\ip\theta .\eeq
Since $\OOG^3(M)=\OO^3(M)^G\oplus\big(\mfgs\otimes\OO^1(M)\big)^G$, 
the 3-curvature decomposes as 
\[ \etag=\alpha+\eta ,\]
where $\alpha\in\OO^3(M)^G$ and $\eta\in\big(\mfgs\otimes\OO^1(M)\big)^G$.
From \eqref{3}, we obtain 
\[ \pi^*(\alpha+\eta)=\dg(\bb+\ff)=d\bb-\widehat{\xi}\ip\bb+d\ff \] 
and it thus follows that
\begin{gather}
\pi^*\alpha=d\bb \label{11} \\
\pi^*\eta=d\ff-\widehat{\xi}\ip\bb \label{12} 
.\end{gather}

We will need a few lemmas.

%

\begin{lem}\label{22}
We have $\eqcls{\BS(\source^*d\ff)-\zeta}=0$ in $H^3_{\DR}\big((\gm\rtimes G)\simplicial\big)$.
\end{lem}

\begin{proof}
Since $\source: \gm\xto{}X$ is $G$-equivariant, from \eqref{eq:BSG} and \eqref{19}, we obtain 
\beq{23} \BS(\source^*d\ff)=(\source\times\id)^*\BS(d\ff)=(\source\times\id)^*d\lambda=d\lambdap ,\eeq
where $\lambdap=(\source\times\id)^*\lambda\in\OO^1(\gm\times G)$. 
More explicitly, $\forall(v_x,w_y)\in T_{(x,y)}\gm, \ \xi\in \mfg, \ g\in G$, we have
\[ \lambdap((v_x,w_y),L_{g*}\xi)=\lambda(w_y,L_{g*}\xi)=\ff(\Ad_g\xi)(y) .\]
The multiplication in the groupoid $\gm\rtimes G\toto X$ is defined by 
\beq{gmgx} \big((x,y),g\big)\cdot\big((x',y'),h\big) = \big((x,y'\act g\inv),gh\big) ,\eeq
provided $y\act g=x'$, where $(x,y),(x',y')\in\gm(\isomorphism X\times_M X)$ and $g,h\in G$.
It thus follows that
\begin{multline*} 
(\delr\lambdap)\big((v_x,w_y,L_{g*}\xi),(v'_{x'},w'_{y'},L_{h*}\eta)\big) \\ 
=\lambdap(v_x,w_y,L_{g*}\xi)-\lambdap\big((v_x,w_y,L_{g*}\xi)\cdot
(v'_{x'},w'_{y'},L_{h*}\eta)\big)+\lambdap(v'_{x'},w'_{y'},L_{h*}\eta)
.\end{multline*}
But \eqref{gmgx} implies that
\[ (v_x,w_y,L_{g*}\xi)\cdot (v'_{x'},w'_{y'},L_{h*}\eta)=(v''_x, w''_{y'\act g\inv},
L_{gh*}(\Ad_{h\inv}\xi+\eta)) ,\]
where $v''_x$ and $w''_{y'\act g\inv}$ are tangent vectors of $X$
at $x$ and ${y'\act g\inv}$ respectively.
Hence we get
\begin{align*} 
& \lambdap\big(  (v_x,w_y,L_{g*}\xi)\cdot (v'_{x'},w'_{y'},L_{h*}\eta) \big) && \\ 
=\;& \ff(\Ad_{gh}(\Ad_{h\inv}\xi+\eta))(y'\act g\inv) && \\ 
=\;& \ff(\Ad_{g\inv}\Ad_{gh}(\Ad_{h\inv}\xi+\eta))(y') 
&& \text{since $\ff$ is $G$-equivariant} \\ 
=\;& \ff(\xi)\;(y')+\ff(\Ad_h\eta)\;(y') && \text{since $\ff$ is linear in $\mfg$} .
\end{align*}
Thus, we have
\begin{align*}
& (\delr\lambdap)\big((v_x,w_y,L_{g*}\xi), (v'_{x'},w'_{y'},L_{h*}\eta)\big) \\
=\;& \ff(\Ad_g\xi)(y)-\big(\ff(\xi)\;(y')+\ff(\Ad_h\eta)\;(y') \big)+\ff(\Ad_h\eta)(y') \\
=\;& \ff(\Ad_g\xi)(y)-\ff(\xi)(y') \\
=\;& \ff(\xi)\;(y\act g)-\ff(\xi)\;(y') \\
=\;& \ff(\xi)\;(x')-\ff(\xi)\;(y') \\
=\;& -(\del\ff)(\xi)\;(x',y') 
.\end{align*}
Now, set $z=(x,y)$ and $z'=(x',y')\in\gm=X\times_M X$ and choose
$\ztilde$ and $\ztilde'\in\gmt$ such that $\pp(\ztilde)=z$ and 
$\pp(\ztilde')=z'$. Moreover, take $v_\ztilde\in T_\ztilde \gmt$ and
$v_{\ztilde'}\in T_{\ztilde'}\gmt$ such that
$p_{G*}v_\ztilde=(v_x,w_y)$ and
$p_{G*}v_{\ztilde'}=(v'_{x'},w'_{y'})$.
Then 
\begin{align*} 
& p_{G}^*(\delr\lambdap)\big((v_\ztilde, L_{g*}\xi),(v_{\ztilde'}, ,L_{h*}\eta) \big) && \\
=\;& (\delr\lambdap)\big((v_x,w_y,L_{g*}\xi),(v'_{x'},w'_{y'},L_{h*}\eta)\big) && \\ 
=\;& -(\del\ff)(\xi)\;(x',y') && \\
=\;& -p^*_{G}((\del\ff)(\xi))\;(\ztilde') && \\ 
=\;& -(\delt\ff)(\xi)\;(\ztilde') && \\ 
=\;& \widehat{\xi}_{\ztilde'}\ip\theta && \text{by \eqref{9}} \\ 
=\;& (p_{G}^*\zeta)\big((v_\ztilde, L_{g*}\xi),(v_{\ztilde'},L_{h*}\eta)\big) 
&& \text{by \eqref{15}} 
.\end{align*}
Hence \beq{25} \delr\lambdap=\zeta .\eeq
From \eqref{23} and \eqref{25}, it follows that 
\[ \BS(\source^*d\ff)-\zeta=d\lambdap-\delr\lambdap .\] 
Therefore we have
$\eqcls{\BS(\source^*d\ff)-\zeta}=0$ in $H^3_{\DR}\big((\gm\rtimes G)\simplicial\big)$. 
\end{proof}

Let $\qq\in\big(\mfgs\otimes\OO^1(X)\big)^G$ be defined by 
$\qq(\xi)=\widehat{\xi}\ip\bb$, $\forall\xi\in\mfg$.

\begin{lem}\label{lem:4.4}
$\delr\bb-\pr^*\omega= (\sourcer)^*\bb-(\source\rond\pr)^*\bb=-\BS(\source^*\qq)$
\end{lem}

\begin{proof}
The source and target maps of the groupoid 
$\gm\rtimes G\toto X$ are given, respectively, by
$\targetr(\gam,g)=\target(\gam)$ and $\sourcer(\gam,g)=\source(\gam)\act g$.
Using \eqref{888}, we obtain 
\begin{align*}
\delr\bb-\pr^*\omega 
=& (\sourcer)^*\bb-(\targetr)^*\bb-\pr^*(\del\bb) \\ 
=& (\sourcer)^*\bb-(\targetr)^*\bb-\pr^*(\source^*\bb-\target^*\bb) \\ 
=& (\sourcer)^*\bb-(\source\rond\pr)^*\bb 
.\end{align*}
This proves the first equality.

For the second equality, note that, by \eqref{eq:BSG} and Lemma~\ref{lem:4.3}, we have 
\[ \BS(\source^*\qq) = (\source\times 1)^*\BS(\qq) = (\source\times 1)^*(\target_0^* B-\source_0^*B) .\]
Here $\target_0$ and $\source_0$ are the source and target maps
of the transformation groupoid $X\rtimes G\toto X$.
It is clear that $\source_0\rond(\source\times 1)=\sourcer$ and 
$\target_0\rond(\source\times 1)=\source\rond\pr$. 
Thus the second equality follows.
\end{proof}

\begin{proof}[Proof of Theorem~\ref{26}] 
Using \eqref{11}, we get 
\beq{A1}
\mu(\etag)=\pih^*\BS(\alpha+\eta)
=\pi^*\alpha+\pih^*\BS(\eta)
=d\bb+\pih^*\BS(\eta) 
.\eeq
And, since $\pi\rond \source :\gm\to M$ is a $G$-equivariant map, we get
\begin{align}
\pih^*\BS(\eta)  
=& ((\pi\rond\source)\times\id)^*\BS(\eta) && \nonumber \\ 
=& \BS\big((\pi\rond\source)^*\eta\big) && \text{by \eqref{eq:BSG}} \nonumber \\ 
=& \BS(\source^*\pi^*\eta) && \nonumber \\ 
=& \BS(\source^*(d\ff-\qq)) && \text{by \eqref{12}} \nonumber \\ 
=& \BS(\source^*d\ff)-\BS(\source^*\qq) . && \label{A2}
\end{align} 
Therefore,
\begin{align*}
\mu(\etag) =& d\bb+\pih^*\BS(\eta) && \text{by \eqref{A1}} \\ 
=& d\bb+\BS(\source^*d\ff)-\BS(\source^*\qq) && \text{by \eqref{A2}} \\
=& d\bb+\BS(\source^*d\ff)+\delr\bb-\pr^*\omega && \text{by Lemma~\ref{lem:4.4}}
.\end{align*} 
Hence, by Lemma~\ref{22}, 
\[ \mu \eqcls{\etag}
=\eqcls{d\bb+\delr\bb}+\eqcls{\BS(\source^*d\ff)-\pr^*\omega}
=\eqcls{\zeta-\pr^*\omega} .\qedhere\]
\end{proof}


\bibliographystyle{hamsplain}
\bibliography{EGbiblio}

\end{document}